\documentclass{amsart}

\usepackage[T1]{fontenc}
\usepackage{amsfonts}

\usepackage{amsthm,amsmath,amssymb}
\usepackage{mathrsfs}
\usepackage[colorlinks,citecolor=blue,urlcolor=blue]{hyperref}  %%check

\usepackage{enumerate}

\usepackage{tgtermes}
\usepackage{mathptmx} 
\usepackage[scaled=.92]{helvet}

\title{Homomorphisms Between Rings with Infinitesimals and Infinitesimal Comparisons}
\author{Emanuele Bottazzi}
\address{University of Pavia, Italy, \href{mailto:emanuele.bottazzi@unipv.it}{emanuele.bottazzi@unipv.it}, \href{mailto:emanuele.bottazzi.phd@gmail.com}{emanuele.bottazzi.phd@gmail.com}}
\date{\today}

\newtheorem{theorem}{Theorem}[section]
\newtheorem{lemma}[theorem]{Lemma}
\newtheorem{proposition}[theorem]{Proposition}
\newtheorem{corollary}[theorem]{Corollary}
\newtheorem{definition}[theorem]{Definition}

\newcommand{\N}{\mathbb{N}}

\newcommand{\R}{\mathbb{R}}

\newcommand{\hR}{\,\!^\ast\R}
\newcommand{\nhR}{\,\!^{+}\mathbb{\R}}

\newcommand{\blob}{\rule[.2ex]{0.6ex}{0.6ex}}
\newcommand{\fR}{\,\!^{\blob}\mathbb{\R}}
\newcommand{\gioR}{\,\!^{\bullet}\mathbb{\R}}

\renewcommand{\null}{c_{00}}
\newcommand{\frechet}{\mathcal{F}}
\newcommand{\U}{\mathcal{U}}
\renewcommand{\o}{{o(1/n)}}
\begin{document}
	
\maketitle
	
\begin{abstract}
	We examine an argument of Reeder suggesting that the nilpotent infinitesimals in Paolo Giordano's ring extension of the real numbers $\gioR$ are smaller than any infinitesimal hyperreal number of Abraham Robinson's nonstandard extension of the real numbers $\hR$.
	Our approach consists in the study of two canonical order-preserving homomorphisms taking values in $\gioR$ and $\hR$, respectively, and whose domain is Henle's extension of the real numbers in the framework of ``non-nonstandard'' analysis.
	The existence of a nonzero element in Henle's ring that is mapped to $0$ in $\gioR$ while it is seen as a nonzero infinitesimal in $\hR$ suggests that some infinitesimals in $\hR$ are smaller than the infinitesimals in $\gioR$.
	We argue that the apparent contradiction with the conclusions by Reeder is only due to the presence of nilpotent elements in $\gioR$.
\end{abstract}

\tableofcontents

In \cite{reeder}, Reeder suggested a comparison of the size of the infinitesimals in $\hR$, a field of hyperreal numbers of nonstandard analysis \cite{nsa robinson}, and $\gioR$, Paolo Giordano's ring extension of the real numbers with nihlpotent infinitesimals, whose construction is presented in \cite{giordano}.
Reeder outlined two possible approaches to the problem.
The first is the comparison of the size of the invertible infinitesimals and of the nilpotent infinitesimals in a nonstandard model of $\gioR$, while the second is the study of the order-preserving homomorphisms between the field and the ring.

Following the second line of research, Reeder
proved that there is no non-trivial homomorphism different from the standard part map from Paolo Giordano's extension of the real numbers $\gioR$  to a field of hyperreal numbers of Robinson's nonstandard analysis $\hR$, and that there are non-trivial homomorphisms different from the standard part map from $\hR$ to $\gioR$.
As a consequence of this result, the author suggested that the infinitesimals in $\gioR$ are smaller than the infinitesimals in $\hR$.

Inspired by Reeder, we propose a further comparison between the size of the nilpotent infinitesimals in $\gioR$ and the infinitesimals of $\hR$.
We will carry out this comparison by studying two canonical order-preserving homomorphisms taking values in $\gioR$ and $\hR$, respectively, and whose domain is $\nhR$, Henle's extension of the real numbers in the framework of ``non-nonstandard'' analysis \cite{henle}.
We will show that there exists a nonzero infintesimal in Henle's ring whose canonical image in $\gioR$ is zero, while its canonical image in $\hR$ is a positive infinitesimal.
The existence of a nonzero element in Henle's ring that is mapped to $0$ in $\gioR$ while it is seen as a nonzero infinitesimal in $\hR$ suggests that some infinitesimals in $\hR$ are smaller than the infinitesimals in $\gioR$.
We believe that the apparent contradiction with Reeder is only due to the presence of nilpotent elements in $\gioR$.

In Section \ref{section review} we will briefly review the definition of Henle's ring $\nhR$, of the ultrapower construction of a field of hyperreal numbers $\hR$, and of Giordano's ring extension of the real numbers $\gioR$.
The core of the paper consists of Section \ref{section homomorphisms}, where we will discuss some canonical homomorphisms between the extensions of the real numbers introduced in Section \ref{section review}.
Our main result, namely the existence of a nonzero element $r$ in Henle's ring that is ``too small'' to be registered as nonzero in $\gioR$ while it is seen as a nonzero infinitesimal in $\hR$, will be presented at the end of this section. % \ref{section main}.
The paper concludes with an attempt at reconciling our result and the conclusions drawn out by Reeder.

\section{Rings with infinitesimals}\label{section review}

Let us recall the definition of infinitesimal and finite elements in ordered rings that extend the real line.

\begin{definition}
	Let $R \supset \R$ be an ordered ring.
	We say that $r \in R$ is an infinitesimal iff for all $n \in \N$, $-1/n \leq r \leq 1/n$.
	We say that $r \in R$ is finite iff there exists $n \in \N$ such that $-n \leq r \leq n$, and we will denote by $R_F$ the subring of finite elements of $R$.
\end{definition} 

\subsection{Fields of hyperreal numbers of Robinson's nonstandard analysis.}

There are many approaches for the introduction of fields of hyperreal numbers of Robinson's nonstandard analysis.
For the purpose of our infinitesimal comparisons, it is convenient to describe them as ultrapowers of $\R^\N$ with respect to nonprincipal ultrafilters.
We recall their definition mostly to fix the notation.

\begin{definition}
	Let $\U \supset \frechet$ be a nonprincipal ultrafilter.
	A field of hyperreal numbers of Robinson's nonstandard analysis is defined as $\hR = \R^\N/\U$.
\end{definition}

The sum, the product and the order over $\hR$ are defined by means of the quotient.
Explicitly, we have
\begin{itemize}
	\item $[x]_\U+[y]_\U = [x+y]_\U$;
	\item $[x]_\U\cdot[y]_\U = [xy]_\U$;
	\item $[x]_\U \leq_\U [y]_\U$ if and only if $\{ n\in \N : x_n \leq y_n \} \in \U$.
\end{itemize}

The latter definition entails that the order $\leq_{\U}$ is total, since the properties of ultrafilters ensure that if $\{ n\in \N : x_n \leq y_n \} \not\in \U$, then $\{ n\in \N : y_n \leq x_n \} = \{ n\in \N : x_n \leq y_n \}^c \in \U$.

The relevant algebraic properties of $\hR$ are well-known and summarized below.

\begin{lemma}
	$\hR$ is a non-archimedean linearly ordered field.
\end{lemma}

For an in-depth presentation of the hyperreal numbers of nonstandard analysis, we refer for instance to \cite{goldblatt}.

\subsection{The ring extension of the real numbers in Henle's non-nonstandard analysis.}

In \cite{henle}, Henle proposed a way of constructing real infinitesimals with an approach similar to that of nonstandard analysis, but without relying on the axiom of choice.
Henle's ring extension of the real numbers is obtained as a quotient of the space of real sequences by the Frechet filter.

\begin{definition}
	Recall that the Frechet filter is defined as $\frechet = \{ A \subseteq \N : A \text{ is cofinite}\}$.
	The ring extension of the real numbers in Henle's non-nonstandard analysis is defined as the quotient
	$\nhR = \R^\N/\frechet$.
\end{definition}

The sum, the product and the order over $\nhR$ are defined in analogy to the corresponding operations and relation in $\hR$.
\begin{itemize}
	\item $[x]_\frechet+[y]_\frechet = [x+y]_\frechet$;
	\item $[x]_\frechet\cdot[y]_\frechet = [xy]_\frechet$;
	\item $[x]_\frechet \leq_\frechet [y]_\frechet$ if and only if $\{ n\in \N : x_n \leq y_n \} \in \frechet$.
\end{itemize}

Notice that, in contrast to the order over the hyperreal fields $\hR$, the order over $\nhR$ is not total.
This happens because there are some sets $A \subset \N$ such that $A \not \in \frechet$ and $A^c \not \in \frechet$.

The algebraic properties of $\nhR$ that will be relevant for the purpose of our infinitesimal comparisons are recalled in the next Lemma.

\begin{lemma}
	\begin{enumerate}
		\item $\nhR$ is a partially ordered ring;
		\item there are nonzero infinitesimals in $\nhR$;
		\item $\nhR$ has $0$-divisors;
		\item $\nhR$ has no nilpotent elements.
	\end{enumerate}
\end{lemma}
\begin{proof}
	(1).
	Observe that $\{n \in \N : 1-(-1)^n\leq 1\} \not\in \frechet$ and that $\{n \in \N : 1\leq 1-(-1)^n \} \not\in \frechet$.
	As a consequence, $[1-(-1)^n]_{\frechet}$ is not comparable with $[1]_{\frechet}$.
	
	(2). The element $[1/n]_{\frechet}$ is a nonzero infinitesimal, since for all $m \in \N$ the inequality $0 < 1/n < m$ is eventually satisfied.
	
	(3). Consider $[1-(-1)^n]_{\frechet}$ and $[1+(-1)^n]_{\frechet}$.
	Both are nonzero elements of $\nhR$, but it is readily verified that $[1-(-1)^n]_{\frechet}\cdot[1+(-1)^n]_{\frechet} = 0$.
	
	(4). Suppose that $[x]_{\frechet} \in \nhR$ is different from $0$.
	From the definition of equality in $\nhR$, we deduce that $\{ n \in \N :x_n \not = 0 \} \in \frechet$.
	Moreover, for all $k \in \N$ it holds the equality $\{ n \in \N :x_n^k \not = 0 \} = \{ n \in \N :x_n \not = 0 \}$, so that $[x]^k_{\frechet} \not = 0$, as desired.
\end{proof}

It is possible to define $\nhR$ also as the quotient of the ring $\R^\N$ with respect to the nonprincipal ideal consisting of the eventually vanishing sequences.

\begin{lemma}\label{lemma isomorfismo}
	If we define
	$\null = \{ x \in \R^\N : \exists \overline{n}\in\N \text{ such that } x_n = 0 \text{ for all } n \geq \overline{n} \},$
	then $\null$ is an ideal of the ring $\R^\N$, and $\R^\N / \null$ is an ordered ring isomorphic to $\nhR$.
\end{lemma}

\subsection{Paolo Giordano's ring extension of the real numbers with nilpotent infinitesimals.}

We follow the definition of Paolo Giordano's ring extension of the real numbers with nilpotent infinitesimals as exposed in \cite{giordano-katz}.

\begin{definition}
	Define $\R^\N_B$ as the ring of bounded sequences.
	The set $\o \subset \R^\N_B$ defined as $\o = \{ f \in \R^\N : \lim_{n \rightarrow \infty} n f(n) = 0\}$ is an ideal of $\R^\N_B$.
	Define the ring $\fR$ as the quotient $\fR = \R^\N_B/\o$.
\end{definition}

The sum, the product and the order over $\fR$ are defined by means of the quotient.
Explicitly, we have
\begin{itemize}
	\item $[x]_\o+[y]_\o = [x+y]_\o$;
	\item $[x]_\o\cdot[y]_\o = [xy]_\o$ whenever $x$ and $y$ are bounded sequences;
	\item $[x]_\o \leq_\o [y]_\o$ if and only if there exists $z \in \o$ and there exists $\overline{n}\in\N$ such that $x_n \leq y_n + z_n$ for all $n \geq \overline{n}$.
\end{itemize}

The ring $\fR$ has a different algebraic structure than $\nhR$, due to the presence of nilpotent elements.

\begin{lemma}
	\begin{enumerate}
		\item $\fR$ is a partially ordered ring;
		\item there are nonzero infinitesimals in $\fR$;
		\item $\fR$ has nilpotent elements and, as a consequence, it has $0$-divisors.
	\end{enumerate}
\end{lemma}
\begin{proof}
	(1).
	%$\fR$ is a ring, since $\o$ is an ideal of $\R^\N_B$.
	The order over $\fR$ is partial, since for instance $[1-(-1)^n]_{\o}$ is not comparable with $[1]_{\o}$.
	
	(2) and (3). The element $[1/n]_{\o}$ is an infinitesimal, since for all $m \in \N$ the inequalities $0<1/n < m$ are eventually satisfied.
	Moreover, $[1/n]_{\o}^2 = [1/n^2]_{\o}$ and, since $1/n^2 \in \o$, $[1/n]_{\o}^2 = 0$, so that $[1/n]_{\o}$ is a nilpotent element of $\fR$.
\end{proof}

Paolo Giordano's ring with infinitesimals is obtained as a subring of $\fR$.

\begin{definition}
	We say that $x\in\R^\N$ is a little-oh polynomial if and only if
	$$
	x_{n}-r+\sum_{i=1}^{k}\alpha_{i}\cdot\frac{1}{n^{a_{i}}} \in \o
	$$
	for some $k\in\N$,
	$r,\alpha_{1},\dots,\alpha_{k}\in\R$, $a_{1},\dots,a_{k}\in\R_{\ge0}$.
	Denote by $\R_{o}$ the set of little-oh polynomials.
	Paolo Giordano's ring extension of the real numbers with nilpotent infinitesimals is defined as
	$\gioR = \R_{o}/\o$.
\end{definition}

\begin{lemma}
	\begin{enumerate}
		\item $\gioR$ is a linearly ordered subring of $\fR$;
		\item there are nonzero infinitesimals in $\gioR$;
		\item every infinitesimal in $\gioR$ is nilpotent.
	\end{enumerate}
\end{lemma}
\begin{proof}
	For the proof of these assertions, we refer to \cite{giordano,giordano-katz}.
\end{proof}

For a matter of commodity, we will usually work with $\fR$ instead of $\gioR$.
Since the order in $\gioR$ is the restriction of the order in $\fR$, this choice does not cause any problems in the infinitesimal comparisons proposed in the next section.

\section{Homomorphisms between rings with infinitesimals}\label{section homomorphisms}

We are ready to discuss some homomorphisms between $\nhR$, $\hR$, $\fR$, and $\gioR$.
Recall that, by Lemma 2.2 of \cite{reeder}, any of such homomorphisms must either be trivial or fix $\R$.
Moreover, by Proposition 3.1 of \cite{reeder}, there are no homomorphisms from $\nhR$ or from $\hR$ to $\fR$ or to $\gioR$, since the latter does not contain infinite numbers.
However, it is possible to overcome this limitation by studying homomorphisms that are defined only on the finite elements of $\nhR$ or $\hR$.

\subsection{A general result about rings with nilpotent elements.}

For the comparison of the infinitesimals in the rings $\nhR$, $\hR$, $\fR$, and $\gioR$, it is important to keep in mind that the image of a nilpotent element through an homomorphism must be either nilpotent or $0$. Thus, if the range of a homomorphism is a ring without nilpotent elements, the image of a nilpotent element must be zero.

\begin{proposition}\label{proposition nil}
	Let $R$ be a ring, and $r \in R$ be a nilpotent element.
	Let also $R'$ be a ring without nilpotent elements.
	Then for all ring homomorphisms $\Psi : R \rightarrow R'$, $\Psi(r) = 0$.
\end{proposition}
\begin{proof}
	Since $R'$ has no nilpotent elements, the equality $(\Psi(r))^k = 0$ is satisfied only if $\Psi(r) = 0$.
\end{proof}

Thanks to this result it is possible to show that Reeder's classification of the homomorphisms from $\gioR$ to $\hR$ (Theorem 2.3 of \cite{reeder}) applies also to the homomorphisms from $\gioR$ to $\nhR$.

\begin{corollary}\label{corollary nil}
	The only non-trivial homomorphism from $\gioR$ into $\nhR$ is the standard part map.
	Similarly, the only non-trivial homomorphism from $\gioR$ into $\hR$ is the standard part map.
\end{corollary}
\begin{proof}
	Both results are a consequence of the previous Proposition, of Lemma 2.2 of \cite{reeder}, and of the fact that every infinitesimal in $\gioR$ is nilpotent.
\end{proof}

The converse of Proposition \ref{proposition nil} is false: if $R$ is a ring with nilpotent elements and if $R'$ is a ring without nilpotent elements, there can be homomorphisms from $R'$ into $R$ whose range includes nilpotent elements.
Similarly, the converse of Corollary \ref{corollary nil} is false: in the proof of Theorem 3.2 of \cite{reeder}, it is explicitly defined a non-trivial homomorphism from $\hR_F$ into $\gioR$ different from the standard part map.

\subsection{Canonical homomorphisms with domain $\nhR$.}

We now introduce two non-trivial homomorphisms from Henle's ring to the field of hyperreal numbers and to Paolo Giordano's ring extension of the real numbers.
These homomorphisms are canonical in the sense that they are defined from the algebraic structure of $\nhR$, $\hR$ and $\fR$ as quotients of the rings $\R^\N$ and $\R^\N_B$.

\begin{definition}
	We define
	\begin{itemize}
		\item $i : \nhR_F \rightarrow \fR$ by
		$i [x]_{\frechet} = [x]_{\o}$;
		\item $j : \nhR \rightarrow \hR$ by
		$j [x]_{\frechet} = [x]_{\U}$.
	\end{itemize}
\end{definition}

\begin{lemma}\label{order-preserving}
	$i$ and $j$ are well-defined, surjective (hence non-trivial), and order preserving in the sense that
	$
	[x]_{\frechet} \leq_\frechet [y]_{\frechet}$ implies
	$i [x]_{\frechet} \leq_{\o} i [y]_{\frechet}$ and
	$j [x]_{\frechet} \leq_\U j [y]_{\frechet}$.
\end{lemma}
\begin{proof}
	Recall the isomorphism between $\nhR$ and $\R^\N / \null$.
	If $[x]_\frechet \in \nhR_F$, then $x \in \R^\N_B$, so that $[x]_\o \in \fR$.
	Since $\o \supset \null$, $i$ is well-defined, surjective and order preserving.
	%It is non-trivial, since $i[1/n]_{\frechet} = [1/n]_{\o} \not = 0$.
	%Notice that $i[1/n]_{\o}$ is a nonzero infinitesimal belonging to $\gioR \subset \fR$.
	Similarly, the inclusion $\U \supset \frechet$ entails that $j$ has the desired properties.
\end{proof}

\subsection{Infinitesimal comparisons}

We are now ready to prove that there exists a nonzero $r \in \nhR$ that is ``too small'' to be registered as nonzero in $\gioR$, while it is seen as a nonzero infinitesimal in $\hR$.
Notice that this result does not depend upon the choice of ultrafilter $\U$ used in the definition of $\hR$.

\begin{proposition}\label{mainresult}
	There exists $r \in \nhR_F$ such that $i(r) = 0$ in $\fR$ and $j(r) \not =0$ in $\hR$
\end{proposition}
\begin{proof}
	Let $x \in \R^\N$ be defined as $x_n = 1/n^2$, and let $r = [x]_{\frechet}$.
	Since the sequence $r$ is bounded, $r \in \nhR_F$, so that $i(r)$ is well-defined.
	Since $x \in \o$, $i(r) = 0$ in $\fR$.
	On the other hand, since $x_n \not = 0$ for all $n \in \N$, the Transfer Principle of nonstandard analysis ensures that $j(r) \not = 0$.
\end{proof}

%Notice that, since $1/n^2$ converges to $0$, $[1/n^2]_\frechet$ is an infinitesimal in $\nhR$.
%From the previous Lemma, we deduce that the element $[1/n^2]_\frechet$ is too small to be distinguished from $0$ in $\fR$ (and, as a consequence, also of the ring $\gioR)$, while it is seen as a nonzero infinitesimal in $\hR$.

\section{Conclusion}

\cite{reeder} argues that the nilpotent infinitesimals in Paolo Giordano's ring are smaller than the infinitesimals of Robinson's nonstandard analysis.
The core of the argument is that the only nontrivial homomorphism from $\gioR$ into $\hR$ is the standard part map, so that the infinitesimals in $\gioR$ are too small to be properly represented in $\hR$.
On the other hand, the existence of a non-trivial homomorphism from $\hR$ into $\gioR$ which is different from the standard part map shows that some infinitesimal hyperreals can be represented as nonzero infinitesimals in $\gioR$.
As a consequence, Reeder suggests that the infinitesimals in $\gioR$ are smaller than those of $\hR$.

As an alternative approach, we proposed to exploit the quotient structure of $\gioR$ and of an ultrapower construction of $\hR$ in order to study some canonical homomorphisms between these structures and $\nhR$, the ring extension of real numbers in Henle's non-nonstandard analysis.
It turned out that there exist nonzero infinitesimals in $\nhR$ that are too small to be represented in $\gioR$, while they are seen as nonzero infinitesimals from the point of view of $\hR$.
Following an argument similar to the one exposed by \cite{reeder}, this result seems to support the claim that the infinitesimals in $\hR$ are more fine-grained than those in $\gioR$.

We believe that the apparent contradiction between our result and Reeder's conclusions on the size of the infinitesimals in $\gioR$ and in $\hR$ arises only because the algebraic structure of $\gioR$ is radically different from the algebraic structure of $\hR$: as seen in Proposition \ref{proposition nil}, it is the presence of nilpotent elements in $\gioR$ that prevents the existence of a wide range of non-trivial homomorphisms from $\gioR$ into the field $\hR$.
Consequently, we suggest that a meaningful comparison of the size of the nilpotent infinitesimals in Paolo Giordano's extension of the real numbers with the infinitesimals of a hyperreal field of nonstandard analysis should not rely only on the study of the homomorphisms between the two structures.

\subsection*{Acknowledgements}
Most of the paper was written while the author was visiting the Department of Mathematics of the University of Trento, Italy.
The author thanks an anonymous referee and prof.\ Mikhail G.\ Katz for valuable comments on a previous version of this paper.

\end{document}